\documentclass[A4,11pt]{article}
\usepackage{amsfonts}
\usepackage{amsmath,amsthm}
\usepackage{amssymb,amscd}
\usepackage{blkarray}
\usepackage{enumerate}
\usepackage{color}
\usepackage{cite}
\usepackage{esint}
\usepackage{tikz}
\usepackage{tikzscale}
\usepackage[absolute,overlay]{textpos}
\usepackage{multicol}
\usepackage{framed}
\usepackage{lscape}
\usepackage{caption}
\usepackage{subcaption}
\usepackage[inline]{enumitem}
\usepackage{float}
\usepackage{kbordermatrix}
\usepackage{relsize}
\usepackage{etoolbox}
\usepackage{MnSymbol}
\usepackage{enumitem}
\usepackage{lineno}
\usepackage{pgf}
\usetikzlibrary{calc,fit,matrix,arrows,automata,positioning}

\usepackage[utf8x]{inputenc}
\usepackage[english]{babel}
\usepackage[margins]{trackchanges}
\addeditor{EL}
\let\bbordermatrix\bordermatrix
\patchcmd{\bbordermatrix}{8.75}{4.75}{}{}
\patchcmd{\bbordermatrix}{\left(}{\left[}{}{}
\patchcmd{\bbordermatrix}{\right)}{\right]}{}{}
\patchcmd{\bbordermatrix}{\begingroup}{\begingroup\openup2\jot}{}{}

\makeatletter

\newcommand{\Rmnum}[1]{\expandafter\@slowromancap\romannumeral #1@}
\makeatother
\def\Xint#1{\mathchoice
{\XXint\displaystyle\textstyle{#1}}%
{\XXint\textstyle\scriptstyle{#1}}%
{\XXint\scriptstyle\scriptscriptstyle{#1}}%
{\XXint\scriptscriptstyle\scriptscriptstyle{#1}}%
\!\int}
\def\XXint#1#2#3{{\setbox0=\hbox{$#1{#2#3}{\int}$}
\vcenter{\hbox{$#2#3$}}\kern-.5\wd0}}

\newcommand\x{\times}

\makeatletter
\newcommand{\vast}{\bBigg@{4}}
\newcommand{\Vast}{\bBigg@{5}}
\newcommand{\dashint}{\Xint-}
\makeatother

\def\Xint#1{\mathchoice
{\XXint\displaystyle\textstyle{#1}}%
{\XXint\textstyle\scriptstyle{#1}}%
{\XXint\scriptstyle\scriptscriptstyle{#1}}%
{\XXint\scriptscriptstyle\scriptscriptstyle{#1}}%
\!\int}
\def\XXint#1#2#3{{\setbox0=\hbox{$#1{#2#3}{\int}$}
\vcenter{\hbox{$#2#3$}}\kern-.5\wd0}}

\theoremstyle{definition}
\newtheorem{theorem}{Theorem}[section]

\newtheorem{lemma}[theorem]{Lemma}

\newtheorem{remark}{Remark}[section]

\providecommand{\keywords}[1]
{
 \small	
  \textbf{\textit{Keywords:}} #1
}
\begin{document}

\title{Integrability of the multi-species TASEP with species-dependent rates}
\author{\textbf{Eunghyun Lee\footnote{eunghyun.lee@nu.edu.kz}}\\ {\text{Department of Mathematics, Nazarbayev University}}
                                         \date{}   \\ {\text{Kazakhstan}}   }

\date{}
\maketitle

\begin{abstract}
\noindent Assume that each species $l$ has its own jump rate $b_l$ in the multi-species totally asymmetric simple exclusion process. We show that this model is \textit{integrable} in the sense that the Bethe Ansatz method is applicable  to obtain the transition probabilities for all possible $N$-particle systems with up to $N$ different species.\\ \\
\end{abstract}

\section{Introduction}
The multi-species asymmetric simple exclusion process on $\mathbb{Z}$ is a generalization of the asymmetric simple exclusion process (ASEP) on $\mathbb{Z}$ in the sense that each particle may belong to a different species labelled by an integer $l \in \{1,2,\cdots\}$. Each particle jumps to the right by one step with probability $p$ or to the left by one step with probability $q = 1-p$ after waiting time exponentially distributed with rate 1. If a particle belonging to $l$ tries to jump to the site occupied by a particle belonging to $l'\geq l$, the jump is prohibited but if a particle belonging to $l'$ tries to jump to the site occupied by a particle belonging to $l<l'$, then the jump occurs by interchanging positions.  The transition probabilities and some determinantal formulas for the multi-species ASEP or its special cases were found in \cite{Chatterjee-Schutz-2010,Kuan-2020,Lee-2018,Lee-2020,Tracy-Widom-2013}. Also, for some special initial conditions with a single second class particle, some distributions and their asymptotics were studied in \cite{Ferrari-Ghosal-Nejjar-2019,Nejjar-2020}. More recently, asymptotic behaviors of the second class particles were studied by using the color-position symmetry, see\cite{Borodin-Bufetov}. In fact, the multi-species asymmetric simple exclusion process can be considered  in more general context, that is, the coloured six vertex model \cite{Borodin-Wheeler}. Another direction of generalizing the ASEP and other models studied in the integrable probability  is to make the jump rates  inhomogeneous. It is known that the Bethe Ansatz method is still applicable to some single-species model with inhomogeneous jump rates. The basic idea of using the Bethe Ansatz in the ASEP is from that the generator of the ASEP is a similarity transformation of that of the XXZ quantum spin system. Considering that the Bethe Ansatz is a method to find eigenvalues and eigenvectors of a certain class of quantum spin systems, we use the Bethe Ansatz to find the solution of the forward equation of a certain class of Markov processes, that is, the transition probabilities of the processes. Of course, for some particle models the Bethe Ansatz method cannot be used. For the background of Bethe Ansatz, see \cite{Baxter,Karbach-Muller}. It is known that the Bethe Ansatz is applicable to some generalization of the ASEP.  For example,  the transition probability and the current distribution of the totally asymmetric simple exclusion process (TASEP) with particle-dependent rates were studied in \cite{Rakos-Schutz-2006},  and the transition probabilities and some asymptotic results for the $q$-deformed totally asymmetric zero range process with site-dependent rates were studied in \cite{Wang-Waugh-2016,Lee-Wang-2019}.
 In this paper, we consider the multi-species totally asymmetric simple exclusion processes  with $N$ particles in which particles move to the right and each species $l$ is allowed  to have its own rate $b_l$. Following the notations used in \cite{Lee-2020}, let $X = (x_1,\cdots, x_N) \in \mathbb{Z}^N$ with  $x_1<\cdots <x_N$ represent the positions of particles, and  let $\pi= \pi(1)\pi(2)~\cdots~\pi(N)$ be a permutation of a multi-set $\mathcal{M} = [i_1,\cdots,i_N]$ with elements taken from $\{1,\cdots, N\}$ and $\pi(i)$ represent the species of the $i^{th}$ leftmost particle. Then, a state of an $N$-particle system is denoted by
\begin{equation*}
(X, \pi) = \big(x_1,\cdots, x_N, \pi(1), \cdots, \pi(N)\big).
\end{equation*}
Let us write $P_{(Y,\nu)}(X,\pi;t)$ for the transition probability from $(Y,\nu)$ at $t=0$ to $(X,\pi)$ at a later time $t$. For fixed $X$ and $Y$, $P_{(Y,\nu)}(X,\pi;t)$ is regarded as a matrix element of an $N^N \times N^N$ matrix denoted by $\mathbf{P}_Y(X;t)$ whose columns and rows are labelled by $\nu, \pi = 1\cdots 11, 1\cdots 12, \cdots, N \cdots N$, respectively. Throughout this paper, given an $N^n \times N^n$ matrix, we assume that its rows ($i_1\cdots i_n$) and columns ($j_1\cdots j_n$) are labelled by $1\cdots 1,\,\cdots,\,n\cdots n$ and these labels are listed lexicographically, unless stated otherwise. The main result of this paper is that the multi-species TASEP with species-dependent rates is an integrable model, and we provide an formula analogous to (2.12) in \cite{Lee-2020} by using  the Bethe Ansatz method.
\subsection{Statement of the results}\label{1023pm0527} We first introduce a few objects to state the main theorem.
Define an $N^2 \times N^2$ matrix $\mathbf{R}_{\beta\alpha} = \big[R_{ij,kl} \big]$   with
\begin{equation}\label{625pm724}
R_{ij,kl} = \begin{cases}
S_{\beta\alpha}(i)& ~\textrm{if}~~ij=kl~\textrm{with}~i\leq j;\\[3pt]
-1& ~\textrm{if}~~ij=kl~\textrm{with}~i>j;\\[3pt]
T_{\beta\alpha}(i)&~\textrm{if}~~ij=lk~\textrm{with}~i<j;\\[3pt]
0 &~\textrm{for all other cases},
\end{cases}
\end{equation}
where
\begin{equation}\label{1100am911}
\begin{aligned}
S_{\beta\alpha}(i) =&  -\frac{1 - b_i\xi_{\beta}}{1 - b_i\xi_{\alpha}},&~ T_{\beta\alpha}(i) =& \frac{b_i(\xi_{\beta}-\xi_{\alpha})}{1 - b_i\xi_{\alpha}}, ~~~~\xi_{\alpha},\xi_{\beta} \in \mathbb{C}.
\end{aligned}
\end{equation}
\begin{remark}
The form of  the matrix  (\ref{625pm724}) was obtained by induction via similar arguments to Section 2.1 and 2.2 in \cite{Lee-2018}  which treats a special case, and the motivation of (\ref{1100am911}) is given in Section \ref{0814pm742}. Finding the form of (\ref{625pm724}) with (\ref{1100am911}) is the key idea of this paper. 
\end{remark}
Let $\mathcal{T}_l$ be the simple transposition  which interchanges the number at the $l^{th}$ slot and the number at the $(l+1)^{st}$ slot.   If $\mathcal{T}_l$ maps a permutation $(~\cdots~ \alpha\beta ~\cdots)$ to $(~\cdots~ \beta\alpha ~\cdots)$, we write $\mathcal{T}_l = \mathcal{T}_l(\beta,\alpha)$ when necessary. Corresponding to  a simple transposition $\mathcal{T}_{l}(\beta,\alpha)$, we define $N^N \times N^N$ matrix $\mathbf{T}_{l}(\beta,\alpha)$ by the tensor product of matrices,
\begin{equation}\label{310am36}
\mathbf{T}_{l}(\beta,\alpha) =\underbrace{\mathbf{I}_N ~\otimes~ \cdots~ \otimes~  \mathbf{I}_N}_{(l-1)~\textrm{times}}~ \otimes~ \mathbf{R}_{\beta\alpha}~ \otimes ~ \underbrace{\mathbf{I}_N ~\otimes~ \cdots ~\otimes~  \mathbf{I}_N}_{(N-l-1)~\textrm{times}}
\end{equation}
where $\mathbf{I}_N$ is the $N \times N$ identity matrix. For a permutation $\sigma$ in the symmetric group $\mathcal{S}_N$ written
\begin{equation}\label{853pm35}
\sigma = \mathcal{T}_{i_j}\cdots \mathcal{T}_{i_1} = \mathcal{T}_{i_j}(\beta_j,\alpha_j)\cdots \mathcal{T}_{i_1}(\beta_1,\alpha_1)
\end{equation}
for some $i_1,\cdots, i_j \in \{1,\cdots, N-1\}$, we define
\begin{equation}\label{305am519}
\mathbf{A}_{\sigma} = \mathbf{T}_{i_j}(\beta_j,\alpha_j)~\cdots~\mathbf{T}_{i_1}(\beta_1,\alpha_1).
\end{equation}

Here, $\mathbf{A}_{\sigma}$ is well defined, that is, $\mathbf{A}_{\sigma}$ is unique regardless of the representation of $\sigma$ by simple transpositions. This \textit{well-definedness} is due to  the following lemma.
\begin{lemma}\label{509pm519} The following consistency relations are satisfied.
\begin{itemize}
  \item [(a)] $ \mathbf{T}_i(\beta,\alpha)\mathbf{T}_j(\delta,\gamma) = \mathbf{T}_j(\delta,\gamma)\mathbf{T}_i(\beta,\alpha) ~~~ \textrm{if}~|i - j| \geq 2.$
  \item [(b)] $\mathbf{T}_i(\gamma,\beta)\mathbf{T}_j(\gamma,\alpha)\mathbf{T}_i(\beta,\alpha) = \mathbf{T}_j(\beta,\alpha) \mathbf{T}_i(\gamma,\alpha)\mathbf{T}_j(\gamma,\beta) ~~~\textrm{if}~|i-j| = 1.$
  \item [(c)] $\mathbf{T}_i(\beta,\alpha) \mathbf{T}_i(\alpha,\beta) = \mathbf{I}_{N^N}.$
\end{itemize}
\end{lemma}
The relations in Lemma \ref{509pm519} with $b_l = 1$ for all $l$ are already known for the multi-species ASEP.
\begin{remark}
The definitions of $\mathcal{T}_l$ and $\mathbf{A}_{\sigma}$ are motivated by the arguments for $N=2,3$ in Section 2.1 and 2.2 in \cite{Lee-2018} which treats a special case.
\end{remark}

Let $\mathbf{J}(t)$ be the $N^N \times N^N$ diagonal matrix whose $(\pi,\pi)$-element is given by  $e^{\varepsilon_{\pi\pi}t}$ where
\begin{equation*}
\varepsilon_{\pi\pi} = \varepsilon_{\pi\pi}(\xi_1,\cdots,\xi_N) =\frac{1}{\xi_1} + \cdots + \frac{1}{\xi_N} - \big(b_{\pi(1)} + \cdots +b_{\pi(N)}\big),~~~~\xi_1,\cdots, \xi_N \in \mathbb{C}
\end{equation*}
and let $\mathbf{D}(x_1,\cdots,x_N)$ be the $N^N \times N^N$ diagonal matrix whose $(\pi,\pi)$-element is given by  $b_{\pi(1)}^{x_1}\cdots b_{\pi(N)}^{x_N}$, and $\mathbf{D'}(y_1,\cdots,y_N)$ be the $N^N \times N^N$ diagonal matrix whose $(\nu,\nu)$-element is given by  $b_{\nu(1)}^{-y_1}\cdots b_{\nu(N)}^{-y_N}$ where $y_i$s are the initial positions. In the next theorem,   the integral of a matrix implies that the integral is taken element-wise, and $\dashint$ implies $\frac{1}{2\pi i}\int$.
\begin{theorem}\label{1249pm523}
Let $\mathbf{A}_{\sigma}$ be given as in (\ref{305am519}) and $c$ be a positively oriented circle centered at the origin with radius less than $b_l$ for all $l$ in the complex plane $\mathbb{C}$. Then, the matrix of the transition probabilities  of the multi-species TASEP with species-dependent rates is
\begin{equation}\label{758pm512}
 \mathbf{P}_Y(X;t)~=~\sum_{\sigma\in \mathcal{S}_N}\dashint_c\cdots \dashint_c \mathbf{J}(t)\mathbf{D}(x_1,\cdots,x_N)\mathbf{A}_{\sigma}\mathbf{D'}(y_1,\cdots,y_N)\prod_{i=1}^N\Big(\xi_{\sigma(i)}^{x_i-y_{\sigma(i)}-1}\Big)~ d\xi_1\cdots d\xi_N.
\end{equation}
\end{theorem}
\begin{remark}
If $\nu = 12\,\cdots\, N$, in other words, all $N$ particles belong to different species and the species are initially arranged in ascending order, then the system is the same as the TASEP with particle-dependent  rates studied in \cite{Rakos-Schutz-2006}. Hence, the transition probability $P_{(Y,\,12\,\cdots \,N)}(X,\,12\,\cdots\, N;t)$ can be expressed as a determinant (see Theorem 1 in \cite{Rakos-Schutz-2006}).
\end{remark}

\begin{remark}
Theorem \ref{1249pm523} partially extends (2.12) in \cite{Lee-2020}. In other words, (\ref{758pm512}) with $b_l=1$ is equal to (2.12) in \cite{Lee-2020} with $p=1$.
\end{remark}

The proofs of Lemma \ref{509pm519} and Theorem \ref{1249pm523} are given in the next section.
\section{Proof of Theorem \ref{1249pm523}}\label{621pm527}
In order to prove that the $(\pi,\nu)$-element of the right-hand side of (\ref{758pm512}) is $P_{(Y,\nu)}(X,\pi;t)$, we should show that the $(\pi,\nu)$-element satisfies its forward equation and the initial condition $P_{(Y,\nu)}(X,\pi;0) = \delta_{X,Y}$.
\subsection{Forward equations}\label{0814pm742}
We first study the two-particle systems, which will be  building-blocks for the formulas for $N$-particle systems.
When $x_1 < x_2 -1$, the forward equations of $P_{(Y,\nu)}(x_1,x_2, \pi;t)$ are straightforward because two particles act as \textit{free} particles. Hence, the forward equations of  $P_{(Y,\nu)}(x_1,x_2, \pi;t)$ are
 expressed as
\begin{equation}\label{446am514}
\begin{aligned}
\frac{d}{dt}\mathbf{P}_Y(x_1,x_2;t) =~& \underbrace{\left[
                                        \begin{array}{cccc}
                                          b_1 & 0 & 0 & 0 \\
                                          0& b_1 & 0 & 0 \\
                                          0 & 0 & b_2 & 0 \\
                                          0 & 0 & 0 & b_2 \\
                                        \end{array}
                                      \right]}_{(r_1)}\mathbf{P}_Y(x_1-1,x_2;t) + \underbrace{\left[
                                        \begin{array}{cccc}
                                          b_1 & 0 & 0 & 0 \\
                                          0& b_2 & 0 & 0 \\
                                          0 & 0 & b_1 & 0 \\
                                          0 & 0 & 0 & b_2 \\
                                        \end{array}
                                      \right]}_{(r_2)}\mathbf{P}_Y(x_1,x_2-1;t) \\[5pt]
                                      &\hspace{1cm} -  \left[
                                        \begin{array}{cccc}
                                          b_1+b_1 & 0 & 0 & 0 \\
                                          0& b_1+b_2 & 0 & 0 \\
                                          0 & 0 & b_2 +b_1& 0 \\
                                          0 & 0 & 0 & b_2+b_2 \\
                                        \end{array}
                                      \right]\mathbf{P}_Y(x_1,x_2;t)
\end{aligned}
\end{equation}
where the derivative of the matrix $\mathbf{P}_Y(x_1,x_2;t)$ on  the left-hand side implies the matrix of the derivatives of elements of $\mathbf{P}_Y(x_1,x_2;t)$. The matrices $(r_1)$ and $(r_2)$ account for \textit{probability current} going in the states $(x_1,x_2,\pi)$ by a particle's jump to the right next site which is  empty. On the other hand, when $x_1= x_2 - 1$, if two particles belong to different species, two particles may swap their positions. For example, if the initial state is $(Y,12)$, the system cannot be at $(X,21)$ at any later time $t$. Hence, the forward equation of  $P_{(Y,12)}(x_1,x_1+1,12;t)$ is
\begin{equation*}
\frac{d}{dt}P_{(Y,12)}(x_1,x_1+1,12;t) = b_1P_{(Y,12)}(x_1-1,x_1+1,12;t) - b_2P_{(Y,12)}(x_1,x_1+1,12;t).
\end{equation*}
On the other hand, $P_{(Y,12)}(x_1,x_1+1,21;t) = 0$ for all $t$ because the model is totally asymmetric. If the initial state is $(Y,21)$, the forward equation of $P_{(Y,21)}(x_1,x_1+1,21;t)$ is
\begin{equation*}
\frac{d}{dt}P_{(Y,21)}(x_1,x_1+1,21;t) = b_2P_{(Y,21)}(x_1-1,x_1+1,21;t) - (b_2+b_1)P_{(Y,21)}(x_1,x_1+1,21;t)
\end{equation*}
and the forward equation of $P_{(Y,21)}(x_1,x_1+1,12;t)$ is
\begin{equation*}
\begin{aligned}
\frac{d}{dt}P_{(Y,21)}(x_1,x_1+1,12;t) = & b_1P_{(Y,21)}(x_1-1,x_1+1,12;t) + b_2P_{(Y,21)}(x_1,x_1+1,21;t) \\
 & - b_2P_{(Y,21)}(x_1,x_1+1,12;t).
\end{aligned}
\end{equation*}
Hence, the forward equations of $P_{(Y,\nu)}(x_1,x_1+1, \pi;t)$ are expressed as
\begin{equation}\label{530pm523}
\begin{aligned}
& \frac{d}{dt}\mathbf{P}_Y(x_1,x_1+1;t)  =  \\[5pt]
 &\hspace{1cm}  \left[
                                        \begin{array}{cccc}
                                          b_1 & 0 & 0 & 0 \\
                                          0& b_1 & 0 & 0 \\
                                          0 & 0 & b_2 & 0 \\
                                          0 & 0 & 0 & b_2 \\
                                        \end{array}
                                      \right]\mathbf{P}_Y(x_1-1,x_1+1;t)+   \underbrace{\left[
                                        \begin{array}{cccc}
                                          0 & 0 & 0 & 0 \\
                                          0& 0 & b_2 & 0 \\
                                          0 & 0 & 0 & 0 \\
                                          0 & 0 & 0 & 0 \\
                                        \end{array}
                                      \right]}_{(B)}\mathbf{P}_Y(x_1,x_1+1;t) \\
 &\hspace{1cm}-   \Vast( \underbrace{\left[
                                        \begin{array}{cccc}
                                          0 & 0 & 0 & 0 \\
                                          0& 0 & 0 & 0 \\
                                          0 & 0 & b_2& 0 \\
                                          0 & 0 & 0 & 0 \\
                                        \end{array}
                                      \right]}_{(C)} + \left[
                                        \begin{array}{cccc}
                                          b_1 & 0 & 0 & 0 \\
                                          0& b_2 & 0 & 0 \\
                                          0 & 0 & b_1& 0 \\
                                          0 & 0 & 0 & b_2 \\
                                        \end{array}
                                      \right]\Vast)\mathbf{P}_Y(x_1,x_1+1;t).
\end{aligned}
\end{equation}
Here, the matrix $(B)$ accounts for \textit{probability current} going in the states $(x_1,x_1+1,12)$ by the species-2 particle's jump from the state $(x_1,x_1+1,21)$. Similarly, the matrix $(C)$ accounts for \textit{probability current} going  out of the states $(x_1,x_1+1,21)$ by species-2 particle's jump to the state $(x_1,x_1+1,12)$.  The equations (\ref{446am514}) and  (\ref{530pm523}) imply that if $\mathbf{U}(x_1,x_2;t)$ is a $4\times 4$ matrix whose elements are functions on $\mathbb{Z}^2 \times [0,\infty)$, then the forward equation of $P_{(Y,\nu)}(x_1,x_2, \pi;t)$ for any $x_1<x_2$ is in the form of the $(\pi, \nu)$-element of
\begin{equation*}
\begin{aligned}
\frac{d}{dt}\mathbf{U}(x_1,x_2;t) =~&\left[
                                        \begin{array}{cccc}
                                          b_1 & 0 & 0 & 0 \\
                                          0& b_1 & 0 & 0 \\
                                          0 & 0 & b_2 & 0 \\
                                          0 & 0 & 0 & b_2 \\
                                        \end{array}
                                      \right]\mathbf{U}(x_1-1,x_2;t) + \left[
                                        \begin{array}{cccc}
                                          b_1 & 0 & 0 & 0 \\
                                          0& b_2 & 0 & 0 \\
                                          0 & 0 & b_1 & 0 \\
                                          0 & 0 & 0 & b_2 \\
                                        \end{array}
                                      \right]\mathbf{U}(x_1,x_2-1;t) \\[5pt]
                                      &\hspace{1cm} -  \left[
                                        \begin{array}{cccc}
                                          b_1+b_1 & 0 & 0 & 0 \\
                                          0& b_1+b_2 & 0 & 0 \\
                                          0 & 0 & b_2 +b_1& 0 \\
                                          0 & 0 & 0 & b_2+b_2 \\
                                        \end{array}
                                      \right]\mathbf{U}(x_1,x_2;t)
\end{aligned}
\end{equation*}
subject to the $(\pi, \nu)$-element of
\begin{equation*}
\begin{aligned}
&\left[
                \begin{array}{cccc}
                  b_1 & 0 & 0 & 0 \\
                  0 & b_2 & 0 & 0 \\
                  0 & 0 & b_1 & 0 \\
                  0 & 0 & 0 & b_2 \\
                \end{array}
              \right]\mathbf{U}(x_1,x_1;t) =  \left[
                \begin{array}{cccc}
                  b_1 & 0 & 0 & 0 \\
                  0 & b_1 & b_2 & 0 \\
                  0 & 0 &  0  & 0 \\
                  0 & 0 & 0 & b_2 \\
                \end{array}
              \right]    \mathbf{U}(x_1,x_1+1;t).
              \end{aligned}
\end{equation*}
Now, we extend the argument for  two-particle systems to $N$-particle systems. The matrices $(r_1)$ and $(r_2)$ in (\ref{446am514}) for two-particle systems are generalized to
  \begin{equation*}
\mathbf{r}_i = \underbrace{\mathbf{I}_N ~\otimes~ \cdots~\otimes~ \mathbf{I}_N}_{i-1}~ \otimes~ \mathbf{r}~
 ~\otimes~ \underbrace{\mathbf{I}_N ~\otimes~ \cdots ~\otimes~\mathbf{I}_N}_{N-i},~~~~~i=1,\cdots,N
\end{equation*}
where $\mathbf{r}$ is the diagonal matrix,
\begin{equation*}
\mathbf{r} =  \left[
                                                                                            \begin{array}{ccc}
                                                                                              b_1 &  &  \\
                                                                                               & \ddots &  \\
                                                                                               &  & b_N \\
                                                                                            \end{array}
                                                                                          \right].
\end{equation*}
 The matrix $(B)$ in (\ref{530pm523}) is generalized to an $N^2 \times N^2$ matrix $\mathbf{B} = \big[B_{ij,kl}\big]$ with
\begin{equation*}
B_{ij,kl} =
 \begin{cases}
b_j& ~\textrm{if}~~ij=lk~\textrm{with}~~ i<j;\\
0 &~\textrm{for all other cases},
\end{cases}
\end{equation*} and let
\begin{equation*}
\mathbf{B}_i =  \underbrace{\mathbf{I}_N ~\otimes~ \cdots~\otimes~ \mathbf{I}_N}_{i-1}~ \otimes ~\mathbf{B}~\otimes ~\underbrace{\mathbf{I}_N~ \otimes~ \cdots ~\otimes~\mathbf{I}_N}_{N-i-1}.
\end{equation*}
The matrix $(C)$ in (\ref{530pm523}) is generalized to $N^2 \times N^2$ matrix $\mathbf{C} = \big[C_{ij,kl}\big]$ with
\begin{equation*}
C_{ij,kl} =
 \begin{cases}
b_i& ~\textrm{if}~~ij=kl~\textrm{with}~~ i<j;\\
0 &~\textrm{for all other cases},
\end{cases}
\end{equation*} and let
\begin{equation*}
\mathbf{C}_i =  \underbrace{\mathbf{I}_N ~\otimes~ \cdots~\otimes~ \mathbf{I}_N}_{i-1}~ \otimes ~\mathbf{C}~\otimes ~\underbrace{\mathbf{I}_N~ \otimes~ \cdots ~\otimes~\mathbf{I}_N}_{N-i-1}.
\end{equation*}
All forward equations of $P_{(Y,\nu)}(x_1,\cdots,x_N, \pi;t)$ may be expressed as a matrix equation. For example, if $x_i < x_{i+1}-1$ for all $i$, then the forward equation of $P_{(Y,\nu)}(x_1,\cdots,x_N, \pi;t)$ is the $(\pi,\nu)$-element of
\begin{equation}\label{332pm529}
\begin{aligned}
 \frac{d}{dt}\mathbf{P}_Y(x_1,\cdots, x_N;t) =~& \mathbf{r}_1\mathbf{P}_Y(x_1-1,x_2,\cdots,x_N;t) + \cdots  \\
&\hspace{0.5cm}+ \mathbf{r}_N\mathbf{P}_Y(x_1,\cdots,x_{N-1},x_N-1;t) - (\mathbf{r}_1+ \cdots +\mathbf{r}_N)\mathbf{P}_Y(x_1,\cdots, x_N;t)
\end{aligned}
\end{equation}
and if $x=x_i = x_{i+1}-1$ and $x_j < x_{j+1}-1$ for all $j\neq i$,
\begin{equation}\label{333pm529}
\begin{aligned}
& \frac{d}{dt}\mathbf{P}_Y(x_1,\cdots,x_{i-1},x,x+1,x_{i+2},\cdots,x_N;t) = \\
&\hspace{0.2cm}+ \sum_{j \neq i+1} \mathbf{r}_{j}\mathbf{P}_Y(x_1,\cdots,x_{j-1},x_j-1, x_{j+1},\cdots,x_N;t)+ \mathbf{B}_i\mathbf{P}_Y(x_1,\cdots,x_{i-1},x,x+1,x_{i+2},\cdots,x_N;t)  \\
&\hspace{0.2cm}  -\sum_{j \neq i} \mathbf{r}_{j}\mathbf{P}_Y(x_1,\cdots,x_{j-1},x,x+1,x_{j+2},\cdots,x_N;t) - \mathbf{C}_i\mathbf{P}_Y(x_1,\cdots,x_{i-1},x,x+1,x_{i+2},\cdots,x_N;t).
\end{aligned}
\end{equation}
For other configurations of $(x_1,\cdots, x_N)$, the form of the matrix of the forward equations may be different from (\ref{332pm529}) and (\ref{333pm529}). However, as in other Bethe Ansatz applicable models, if $\mathbf{U}(x_1,\cdots,x_N;t) = \big[U_{\pi\nu}(x_1,\cdots,x_N;t)\big]$ is an $N^N \times N^N$ matrix whose elements $U_{\pi\nu}(x_1,\cdots,x_N;t)$ are functions on $\mathbb{Z}^N \times [0,\infty)$, then   the forward equation of $P_{(Y,\nu)}(x_1,\cdots, x_N,\pi;t)$ for any $x_1<\cdots <x_N$ is in the form of the $(\pi,\nu)$-element of
\begin{equation}\label{832pm515}
\begin{aligned}
\frac{d}{dt}\mathbf{U}(x_1,\cdots,x_N;t) = ~& \sum_{j=1}^N \mathbf{r}_{j}\mathbf{U}(x_1,\cdots,x_{j-1},x_j-1, x_{j+1},\cdots,x_N;t)  -(\mathbf{r}_1+ \cdots +\mathbf{r}_N)\mathbf{U}(x_1,\cdots,x_N;t)
\end{aligned}
\end{equation}
subject to the $(\pi,\nu)$-element of
\begin{equation}\label{1019pm515}
\begin{aligned}
&\mathbf{r}_{i+1}\mathbf{U}(x_1,\cdots,x_{i-1},x_i,x_i,x_{i+2},\cdots,x_N;t) = \\
& \hspace{2.5cm}(\mathbf{B}_i + \mathbf{r}_i - \mathbf{C}_i)\mathbf{U}(x_1,\cdots,x_{i-1},x_i,x_i+1,x_{i+2},\cdots,x_N;t)
\end{aligned}
\end{equation}
for all $i=1,\cdots, N-1$.
\subsection{Solutions of the forward equations via Bethe Ansatz}
The $(\pi,\nu)$-element of (\ref{832pm515})  is
\begin{equation*}
\begin{aligned}
&\frac{d}{dt}{U}_{\pi\nu}(x_1,\cdots,x_N;t) = \\
& \hspace{1.5cm} \sum_{j=1}^N b_{\pi(j)}U_{\pi\nu}(x_1,\cdots,x_{j-1},x_j-1, x_{j+1},\cdots,x_N;t) -\big(b_{\pi(1)}+\cdots +b_{\pi(N)}\big)U_{\pi\nu}(x_1,\cdots,x_N;t).
\end{aligned}
\end{equation*}
Assume the separation of variables to write ${U}_{\pi\nu}(x_1,\cdots,x_N;t) = {U}_{\pi\nu}(x_1,\cdots,x_N)T(t)$. Then, the equation of the spatial variables is
\begin{equation}\label{156am56}
\begin{aligned}
&  \varepsilon {U}_{\pi\nu}(x_1,\cdots,x_N) =   b_{\pi(1)} {U}_{\pi\nu}(x_1-1,x_2,\cdots,x_N) + \cdots   \\ & \hspace{2cm}+ b_{\pi(N)}{U}_{\pi\nu}(x_1,\cdots,x_{N-1},x_N-1) -\,\big(b_{\pi(1)} + \cdots + b_{\pi(N)}\big) {U}_{\pi\nu}(x_1,\cdots, x_N)
\end{aligned}
\end{equation}
for some constant $\varepsilon$ with respect to $t,x_1,\cdots, x_N$.
Then, we observe that for any  $\sigma \in \mathcal{S}_N$,
\begin{equation*}
\prod_{i=1}^N\Big(b_{\pi(i)}\xi_{\sigma(i)}\Big)^{x_i},~~~~\xi_1,\cdots,\xi_N \in \mathbb{C}
\end{equation*}
solves (\ref{156am56}) if and only if
\begin{equation}\label{219am56}
\varepsilon = \frac{1}{\xi_1} + \cdots + \frac{1}{\xi_N} - \big(b_{\pi(1)} + \cdots + b_{\pi(N)}\big).
\end{equation}
Based on the observation in the above, assume that the matrix $\mathbf{U}(x_1,\cdots,x_N;t)$ is invertible and it is decomposed as
\begin{equation*}
\mathbf{U}(x_1,\cdots,x_N;t) = \mathbf{J}(t)\mathbf{U}(x_1,\cdots,x_N)
\end{equation*}
where $\mathbf{J}(t) = \big[J_{\pi\nu}(t) \big] $ is an
$N^N \times N^N$ diagonal matrix where $J_{\pi\pi}(t)$ are functions of time only. Hence, from (\ref{832pm515}), we obtain
\begin{equation}\label{251am56}
\begin{aligned}
\mathbf{J}(t)^{-1}\bigg(\frac{d}{dt}\mathbf{J}(t)\bigg) =~& \Big(\mathbf{r}_1\mathbf{U}(x_1-1,x_2,\cdots,x_N) + \cdots  + \mathbf{r}_N\mathbf{U}(x_1,\cdots,x_{N-1},x_N-1) \\
&\hspace{1cm}
- (\mathbf{r}_1+ \cdots +\mathbf{r}_N)\mathbf{U}(x_1,\cdots, x_N)\Big) \mathbf{U}(x_1,\cdots,x_N)^{-1}.
\end{aligned}
\end{equation}
Both sides of (\ref{251am56}) must be a diagonal matrix $\mathbf{E}= \big[\varepsilon_{\pi\pi}\big] $ whose elements are some constants with respect to $t,x_1,\cdots, x_N$. Thus, we obtain the matrix equation for spatial variables
\begin{equation}\label{1047pm56}
\begin{aligned}
\mathbf{E}\mathbf{U}(x_1,\cdots,x_N) = ~& \mathbf{r}_1\mathbf{U}(x_1-1,x_2,\cdots,x_N) + \cdots  + \mathbf{r}_N\mathbf{U}(x_1,\cdots,x_{N-1},x_N-1) \\
&\hspace{1cm}
- (\mathbf{r}_1+ \cdots +\mathbf{r}_N)\mathbf{U}(x_1,\cdots, x_N)
\end{aligned}
\end{equation}
and the matrix equation for time variable
\begin{equation*}
\frac{d}{dt}\mathbf{J}(t) = \mathbf{J}(t)\mathbf{E} = \mathbf{E} \mathbf{J}(t).
\end{equation*}
\begin{lemma}\label{349pm57}
Let $\mathbf{D}(x_1,\cdots,x_N) = \big[D_{\pi\nu}(x_1,\cdots, x_N)\big]$ be the $N^N \times N^N$ diagonal matrix with
\begin{equation*}
D_{\pi\pi}(x_1,\cdots, x_N)=b_{\pi(1)}^{x_1}\cdots b_{\pi(N)}^{x_N}.
\end{equation*}
Then,  for any $\sigma \in \mathcal{S}_N$,
\begin{equation}\label{408pm57}
\mathbf{U}(x_1,\cdots,x_N) = \mathbf{D}(x_1,\cdots,x_N)\mathbf{A}\prod_{i=1}^N\xi_{\sigma(i)}^{x_i}
\end{equation}
where $\mathbf{A}$ is an arbitrary invertible $N^N \times N^N$ matrix whose elements are constants with respect to $x_1,\cdots, x_N$ is a solution of  (\ref{1047pm56}) if and only if  $\varepsilon_{\pi\pi}$ is given by
\begin{equation}\label{900pm525}
\varepsilon_{\pi\pi} = \frac{1}{\xi_1} + \cdots +\frac{1}{\xi_N} - \big(b_{\pi(1)} + \cdots + b_{\pi(N)}\big).
\end{equation}
\end{lemma}
\begin{proof}
First, we observe that
\begin{equation*}
\begin{aligned}
\mathbf{r}_i\mathbf{D}(x_1,\cdots,x_{i-1},x_i-1,x_{i+1},\cdots, x_N) = \mathbf{D}(x_1,\cdots,x_N)
\end{aligned}
\end{equation*}
because $\mathbf{r}_i$ is a diagonal matrix whose $(\pi,\pi)$-element is $b_{\pi(i)}$. Also, observe that
\begin{equation*}
\mathbf{D}^{-1}(x_1,\cdots,x_N)= \mathbf{D}(x_1^{-1},\cdots,x_N^{-1}),
\end{equation*}
and
\begin{equation*}
\mathbf{D}(x_1,\cdots,x_N)\mathbf{D}(y_1,\cdots,y_N) = \mathbf{D}(x_1+y_1,\cdots,x_N+y_N).
\end{equation*}
Now, we prove the statement.
Suppose that (\ref{408pm57}) is a solution of (\ref{1047pm56}). Substituting (\ref{408pm57}) into (\ref{1047pm56}) and then dividing both sides by $\prod_{i=1}^N\xi_{\sigma(i)}^{x_i}$, then
\begin{equation*}
\begin{aligned}
 \mathbf{E}\mathbf{D}(x_1,\cdots, x_N)\mathbf{A} =~& \mathbf{r}_1\mathbf{D}(x_1-1,x_2,\cdots, x_N)\mathbf{A}\xi_{\sigma(1)}^{-1} + \cdots \\
& \hspace{1cm} + \mathbf{r}_N\mathbf{D}(x_1,\cdots,x_{N-1}, x_N-1)\mathbf{A}\xi_{\sigma(N)}^{-1}  - (\mathbf{r}_1+ \cdots + \mathbf{r}_N)\mathbf{D}(x_1,\cdots,x_N)\mathbf{A} \\
=~&\mathbf{D}(x_1,\cdots, x_N)\mathbf{A}\xi_{\sigma(1)}^{-1} + \cdots + \mathbf{D}(x_1,\cdots, x_N)\mathbf{A}\xi_{\sigma(N)}^{-1} \\
&\hspace{1cm} -  (\mathbf{r}_1+ \cdots + \mathbf{r}_N)\mathbf{D}(x_1,\cdots,x_N)\mathbf{A}.
\end{aligned}
\end{equation*}
Multiplying by $\mathbf{A}^{-1}\mathbf{D}^{-1}(x_1,\cdots, x_N)$ on both side, we obtain
\begin{equation*}
 \mathbf{E} = \Big(\tfrac{1}{\xi_{\sigma(1)}} +\cdots +  \tfrac{1}{\xi_{\sigma(N)}}\Big)\mathbf{I}_{N^N} - (\mathbf{r}_1+ \cdots + \mathbf{r}_N),
\end{equation*}
and thus, the $(\pi,\pi)$-element of $\mathbf{E}$ is given by (\ref{900pm525}). The second part of the proof can be done via the reverse way of the first part of the proof.
\end{proof}
The previous lemma implies that  the general solution of (\ref{1047pm56}) is given by
\begin{equation}\label{1017pm515}
\mathbf{U}(x_1,\cdots,x_N) = \sum_{\sigma \in \mathcal{S}_N}\mathbf{D}(x_1,\cdots,x_N)\mathbf{A}_{\sigma}\prod_{i=1}^N\xi_{\sigma(i)}^{x_i}.
\end{equation}
\subsection{Boundary conditions}
Now, (\ref{1017pm515}) should satisfy the spatial part of the \textit{boundary condition} (\ref{1019pm515}), that is,
\begin{equation}\label{1019pm516}
\begin{aligned}
&\mathbf{r}_{i+1}\mathbf{U}(x_1,\cdots,x_{i-1},x_i,x_i,x_{i+2},\cdots,x_N) = \\
&\hspace{2cm}(\mathbf{B}_i + \mathbf{r}_i - \mathbf{C}_i)\mathbf{U}(x_1,\cdots,x_{i-1},x_i,x_i+1,x_{i+2},\cdots,x_N)
\end{aligned}
\end{equation}
for $i=1,\cdots, N-1$. Extending the technique used in \cite{Lee-2020}, we will find the formulas of $\mathbf{A}_{\sigma}$ in (\ref{1017pm515}) which satisfy (\ref{1019pm516}).
 Define an $N \times N$ diagonal matrix,
\begin{equation*}
\mathbf{r}^{x_i}:=\left[
                               \begin{array}{ccc}
                                 b_1^{x_i} &  &  \\
                                  & \ddots &  \\
                                  &  & b_N^{x_i} \\
                               \end{array}
                             \right]
\end{equation*}
and recall the definition of $\mathbf{R}_{\beta\alpha}$ in (\ref{625pm724}). Then, we observe that
\begin{equation*}
\begin{aligned}
\mathbf{R}_{\beta\alpha}= -\Big[\big(\mathbf{I}_{N^2} -  \xi_{\alpha}(\mathbf{B} + \mathbf{r} \otimes \mathbf{I}_N - \mathbf{C})\big)\big(\mathbf{r}^{x_i} \otimes \mathbf{r}^{x_i+1}\big)\Big]^{-1}\Big[\big(\mathbf{I}_{N^2} -  \xi_{\beta}(\mathbf{B} + \mathbf{r} \otimes \mathbf{I}_N - \mathbf{C})\big)\big(\mathbf{r}^{x_i} \otimes \mathbf{r}^{x_i+1}\big)\Big].
\end{aligned}
\end{equation*}
\begin{lemma}\label{325am519}
If
\begin{equation}\label{1225am519}
\mathbf{A}_{T_i\sigma} = \Big(\mathbf{I}_{N^{i-1}} \otimes~\mathbf{R}_{\sigma(i+1)\sigma(i)} \otimes \mathbf{I}_{N^{N-i-1}} \Big)\mathbf{A}_{\sigma}
\end{equation}
for all even permutations $\sigma$ and $i=1,\cdots, N-1$,  then (\ref{1019pm516}) is satisfied.
\end{lemma}
\begin{proof}
First, we note that
  \begin{equation*}
  \begin{aligned}
& \mathbf{r}_{i+1} = \underbrace{\mathbf{I}_N ~\otimes ~\cdots~\otimes~ \mathbf{I}_N}_{i-1} ~\otimes~ ( \mathbf{I}_N ~\otimes~ \mathbf{r}) ~\otimes~ \underbrace{\mathbf{I}_N ~\otimes~ \cdots ~\otimes~\mathbf{I}_N}_{N-i-1}, \\
& \mathbf{B}_i + \mathbf{r}_i - \mathbf{C}_i =  \underbrace{\mathbf{I}_N ~\otimes ~\cdots~\otimes~ \mathbf{I}_N}_{i-1} ~\otimes~ (\mathbf{B} + \mathbf{r}~ \otimes~ \mathbf{I}_N - \mathbf{C})
 ~\otimes~ \underbrace{\mathbf{I}_N ~\otimes~ \cdots ~\otimes~\mathbf{I}_N}_{N-i-1},
\end{aligned}
\end{equation*}
and
\begin{equation*}
 \mathbf{D}(x_1,\cdots,x_N) =\mathbf{r}^{x_1}~ \otimes ~\cdots~ \otimes ~\mathbf{r}^{x_N}.
\end{equation*}
Substituting (\ref{1017pm515}) into (\ref{1019pm516}), we obtain
\begin{equation}\label{1122pm518}
\begin{aligned}
& \sum_{\sigma \in \mathcal{S}_N}\Big[\big(\mathbf{r}^{x_1}~ \otimes ~\cdots~\otimes~ \mathbf{r}^{x_{i-1}} ~\otimes ~ \big(\mathbf{r}^{x_i} ~\otimes~ \mathbf{r}^{x_i+1}\big) ~\otimes~ \mathbf{r}^{x_{i+2}}~\otimes ~\cdots~ \otimes~ \mathbf{r}^{x_N} \big)~ - \\
&\hspace{0.5cm}\big(\mathbf{r}^{x_1}~ \otimes ~\cdots~\otimes~ \mathbf{r}^{x_{i-1}} ~\otimes~ \big(\mathbf{B} + \mathbf{r} ~\otimes~ \mathbf{I}_N - \mathbf{C}\big)\big(\mathbf{r}^{x_i} ~\otimes~ \mathbf{r}^{x_i+1}\big)~ \otimes~ \mathbf{r}^{x_{i+2}}~\otimes ~\cdots~ \otimes~ \mathbf{r}^{x_N} \big)\xi_{\sigma(i+1)}\Big] \\
&\hspace{2cm}\times \mathbf{A}_{\sigma}\xi_{\sigma(i)}^{x_i}\xi_{\sigma(i+1)}^{x_i}\prod_{j\neq i,\,i+1}\xi_{\sigma(j)}^{x_j} = \mathbf{0}.
\end{aligned}
\end{equation}
If we express  (\ref{1122pm518}) as a sum over the alternating group $\mathcal{A}_N$
\begin{equation}\label{1204am519}
\begin{aligned}
& \sum_{\sigma \in \mathcal{A}_N}\bigg(\Big[\big(\mathbf{r}^{x_1}~ \otimes ~\cdots~\otimes \mathbf{r}^{x_{i-1}} \otimes \big(\mathbf{r}^{x_i} \otimes \mathbf{r}^{x_i+1}\big) \otimes \mathbf{r}^{x_{i+2}}\otimes \cdots \otimes \mathbf{r}^{x_N} \big) - \\
&\hspace{1.5cm}\big(\mathbf{r}^{x_1}~ \otimes ~\cdots~\otimes \mathbf{r}^{x_{i-1}} \otimes \big(\mathbf{B} + \mathbf{r} \otimes \mathbf{I}_N - \mathbf{C}\big)\big(\mathbf{r}^{x_i} \otimes \mathbf{r}^{x_i+1}\big) \otimes \mathbf{r}^{x_{i+2}}\otimes \cdots \otimes \mathbf{r}^{x_N} \big)\xi_{\sigma(i+1)}\Big]  \mathbf{A}_{\sigma}  \\[5pt]
&\hspace{0.5cm} + \Big[\big(\mathbf{r}^{x_1}~ \otimes ~\cdots~\otimes \mathbf{r}^{x_{i-1}} \otimes \big(\mathbf{r}^{x_i} \otimes \mathbf{r}^{x_i+1}\big) \otimes \mathbf{r}^{x_{i+2}}\otimes \cdots \otimes \mathbf{r}^{x_N} \big) - \\[5pt]
&\hspace{1.5cm}\big(\mathbf{r}^{x_1}~ \otimes ~\cdots~\otimes \mathbf{r}^{x_{i-1}} \otimes \big(\mathbf{B} + \mathbf{r} \otimes \mathbf{I}_N - \mathbf{C}\big)\big(\mathbf{r}^{x_i} \otimes \mathbf{r}^{x_i+1}\big) \otimes \mathbf{r}^{x_{i+2}}\otimes \cdots \otimes \mathbf{r}^{x_N} \big)\xi_{\sigma(i)}\Big]  \mathbf{A}_{T_i\sigma}\bigg)\\
&\hspace{0.5cm}\times~ \xi_{\sigma(i)}^{x_i}\xi_{\sigma(i+1)}^{x_i}\prod_{j\neq i,\,i+1}\xi_{\sigma(j)}^{x_j} = \mathbf{0}.
\end{aligned}
\end{equation}
But, (\ref{1204am519}) is satisfied  if
\begin{equation*}
\begin{aligned}
&\Big[\big(\mathbf{r}^{x_1}~ \otimes ~\cdots~\otimes \mathbf{r}^{x_{i-1}} \otimes \big(\mathbf{r}^{x_i} \otimes \mathbf{r}^{x_i+1}\big) \otimes \mathbf{r}^{x_{i+2}}\otimes \cdots \otimes \mathbf{r}^{x_N} \big) - \\
&\hspace{1.5cm}\big(\mathbf{r}^{x_1}~ \otimes ~\cdots~\otimes \mathbf{r}^{x_{i-1}} \otimes \big(\mathbf{B} + \mathbf{r} \otimes \mathbf{I}_N - \mathbf{C}\big)\big(\mathbf{r}^{x_i} \otimes \mathbf{r}^{x_i+1}\big) \otimes \mathbf{r}^{x_{i+2}}\otimes \cdots \otimes \mathbf{r}^{x_N} \big)\xi_{\sigma(i+1)}\Big]  \mathbf{A}_{\sigma}  \\[5pt]
= ~~& -~\Big[\big(\mathbf{r}^{x_1}~ \otimes ~\cdots~\otimes \mathbf{r}^{x_{i-1}} \otimes \big(\mathbf{r}^{x_i} \otimes \mathbf{r}^{x_i+1}\big) \otimes \mathbf{r}^{x_{i+2}}\otimes \cdots \otimes \mathbf{r}^{x_N} \big) - \\[5pt]
&\hspace{1.5cm}\big(\mathbf{r}^{x_1}~ \otimes ~\cdots~\otimes \mathbf{r}^{x_{i-1}} \otimes \big(\mathbf{B} + \mathbf{r} \otimes \mathbf{I}_N - \mathbf{C}\big)\big(\mathbf{r}^{x_i} \otimes \mathbf{r}^{x_i+1}\big) \otimes \mathbf{r}^{x_{i+2}}\otimes \cdots \otimes \mathbf{r}^{x_N} \big)\xi_{\sigma(i)}\Big]  \mathbf{A}_{T_i\sigma}
\end{aligned}
\end{equation*}
for each even permutation $\sigma$, which is equivalent to  (\ref{1225am519}).
\end{proof}
 In fact,  (\ref{310am36}) and (\ref{305am519}) implies that the assumption of Lemma \ref{325am519} is satisfied, hence (\ref{1017pm515}) with $\mathbf{A}_{\sigma}$ in (\ref{305am519}) satisfies (\ref{1019pm516}).
\subsection{Consistency relations - Proof of Lemma \ref{509pm519}}
Lemma \ref{509pm519}  confirms that  the multi-species TASEP is \textit{integrable} even when the rates are species-dependent.
\subsubsection{Proof of Lemma \ref{509pm519} (a)} It suffices to show that
\begin{equation*}
(\mathbf{R}_{\alpha\beta} ~\otimes~ \mathbf{I}_N ~\otimes~ \mathbf{I}_N)(\mathbf{I}_N ~\otimes~ \mathbf{I}_N ~\otimes~ \mathbf{R}_{\gamma\delta}) = (\mathbf{I}_N ~\otimes~ \mathbf{I}_N ~\otimes~ \mathbf{R}_{\gamma\delta})(\mathbf{R}_{\alpha\beta} ~\otimes~ \mathbf{I}_N ~\otimes~ \mathbf{I}_N).
\end{equation*}
This equality obviously holds because both sides are equal to $\mathbf{R}_{\alpha\beta} ~\otimes~\mathbf{R}_{\gamma\delta}$.\qed
\subsubsection{Proof of Lemma \ref{509pm519} (b) - Yang-Baxter equation}
 It suffices to show
\begin{equation}\label{754pm519}
(\mathbf{R}_{\gamma\beta} ~\otimes~ \mathbf{I}_N)(\mathbf{I}_N ~\otimes~ \mathbf{R}_{\gamma\alpha})(\mathbf{R}_{\beta\alpha} ~\otimes~ \mathbf{I}_N) = (\mathbf{I}_N ~\otimes~ \mathbf{R}_{\beta\alpha})(\mathbf{R}_{\gamma\alpha} ~\otimes ~ \mathbf{I}_N)(\mathbf{I}_N ~\otimes~ \mathbf{R}_{\gamma\beta}).
\end{equation}
  If we re-arrange the columns and the rows of the $N^3 \times N^3$ matrices in (\ref{754pm519}) so that all their labels from the same multi-set $[i,j,k]$ are grouped together, then the matrices in (\ref{754pm519}) become  block-diagonal (See Figure \ref{3400am48}).
  \begin{figure}[H]
 \centering
\includegraphics[width=13cm]{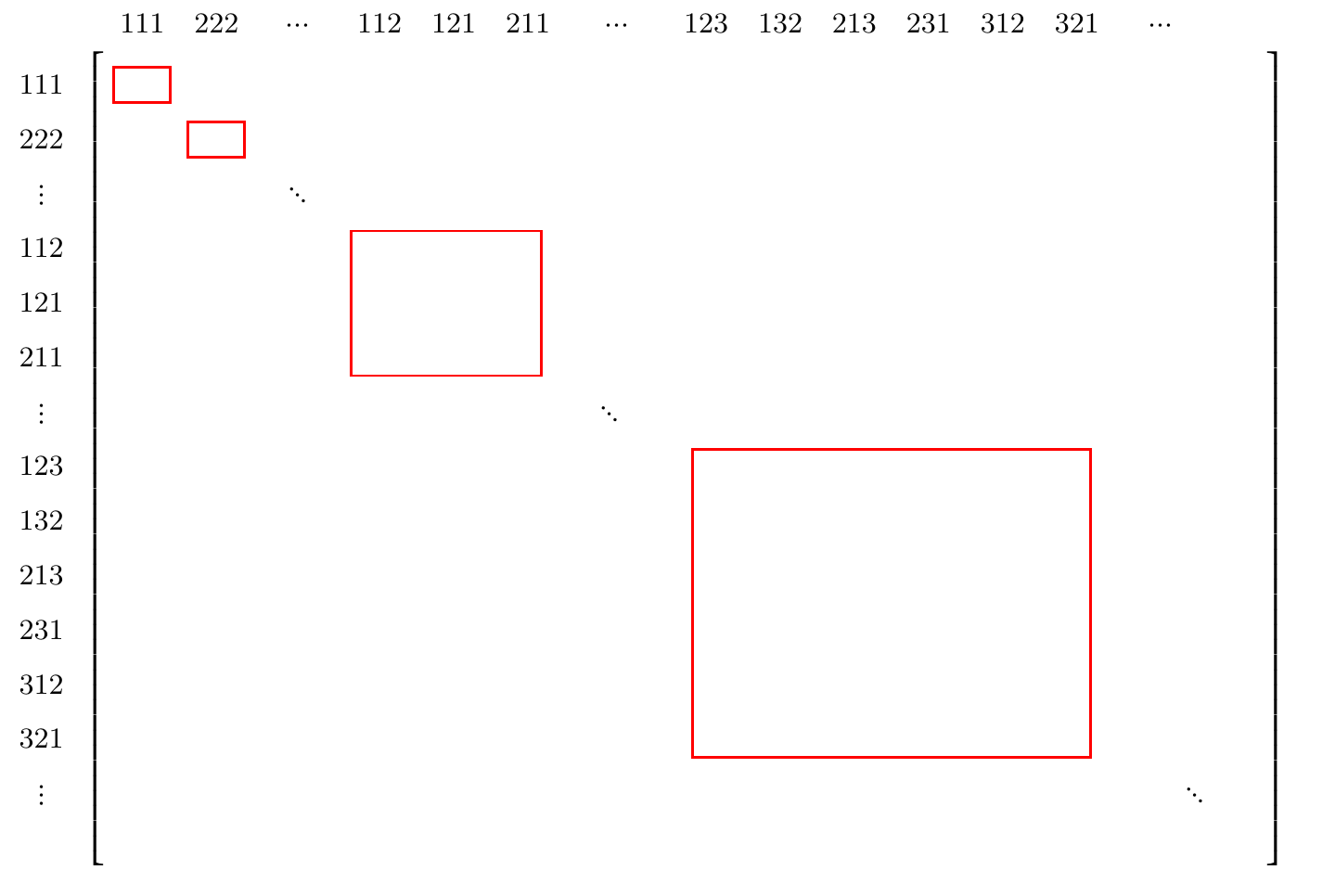}
\caption{\footnotesize{Form of the $N^3 \times N^3$ matrices in (\ref{754pm519}) after re-ordering rows and columns}}
\label{3400am48}
\end{figure}
Let $(\mathbf{R} ~\,\otimes\,~ \mathbf{I}_N)_{[i,j,k]}$  be the sub-matrix of $(\mathbf{R} ~\,\otimes\,~ \mathbf{I}_N)$ whose rows and columns are labelled by the permutations of the multi-set $[i,j,k]$, and similarly, we define $(\mathbf{I}_N\,\otimes\,~ \mathbf{R})_{[i,j,k]}$. Then, in order to show (\ref{754pm519}), it suffices to show
\begin{equation}\label{844pm519}
\begin{aligned}
&(\mathbf{R}_{\gamma\beta} ~\otimes~ \mathbf{I}_N)_{[i,j,k]}(\mathbf{I}_N ~\otimes ~ \mathbf{R}_{\gamma\alpha})_{[i,j,k]}(\mathbf{R}_{\beta\alpha} ~\otimes~ \mathbf{I}_N)_{[i,j,k]} = \\
 &\hspace{2cm}(\mathbf{I}_N ~\otimes~ \mathbf{R}_{\beta\alpha})_{[i,j,k]}(\mathbf{R}_{\gamma\alpha} ~\otimes ~ \mathbf{I}_N)_{[i,j,k]}(\mathbf{I}_N ~\otimes ~\mathbf{R}_{\gamma\beta})_{[i,j,k]}
\end{aligned}
\end{equation}
for each multi-set $[i,j,k]$ whose elements are from $\{1,2,3\}$ because all matrices are block-diagonal matrices in the same form. If $i=j=k$, (\ref{844pm519}) is equivalent to
\begin{equation*}
 S_{\gamma\beta}(i)S_{\gamma\alpha}(i)  S_{\beta\alpha} (i) =   S_{\beta\alpha}(i)  S_{\gamma\alpha} (i) S_{\gamma\beta}(i),
\end{equation*}
which is trivially true. If $i =j >k$, then
(\ref{844pm519}) is equivalent to
\begin{equation*}
\begin{aligned}
\left[
  \begin{array}{ccc}
    S_{\gamma\beta}(k) & T_{\gamma\beta}(k) & 0 \\
    0 &-1  & 0 \\
    0 & 0 &  S_{\gamma\beta}(i) \\
  \end{array}
\right]\left[
         \begin{array}{ccc}
           S_{\gamma\alpha}(i) & 0 & 0 \\
           0 & S_{\gamma\alpha}(k) & T_{\gamma\alpha}(k) \\
           0 & 0 & -1 \\
         \end{array}
       \right]
\left[
  \begin{array}{ccc}
    S_{\beta\alpha} (k) & T_{\beta\alpha} (k) & 0 \\
    0 & -1 & 0 \\
    0 & 0 & S_{\beta\alpha} (i) \\
  \end{array}
\right] \\[5pt]
= \left[
         \begin{array}{ccc}
           S_{\beta\alpha}(i) & 0 & 0 \\
           0 & S_{\beta\alpha}(k) & T_{\beta\alpha}(k) \\
           0 & 0 & -1 \\
         \end{array}
       \right]
\left[
  \begin{array}{ccc}
    S_{\gamma\alpha} (k) &T_{\gamma\alpha} (k) & 0 \\
    0 & -1 & 0  \\
    0 & 0 & S_{\gamma\alpha} (i) \\
  \end{array}
\right]  \left[
         \begin{array}{ccc}
           S_{\gamma\beta}(i) & 0 & 0 \\
           0 &S_{\gamma\beta}(k) & T_{\gamma\beta}(k) \\
           0 & 0 &  -1\\
         \end{array}
       \right],
\end{aligned}
\end{equation*}
 which can be easily verified by direct computation. Similarly, the other two cases of (\ref{844pm519}) for $i=j <k$ and for the case that all $i,j,k$ are distinct can be verified by direct computation.\qed
 \subsubsection{Proof of Lemma \ref{509pm519} (c)} It suffices to show that $\mathbf{R}_{\beta\alpha}  \mathbf{R}_{\alpha\beta} = \mathbf{I}_{N^2}$. Let us re-arrange the rows  and the columns  in the same way as in the proof of Lemma \ref{509pm519} (b) to make  $\mathbf{R}_{\beta\alpha} $ and $\mathbf{R}_{\alpha\beta} $ block-diagonal. Then, each block on the diagonal of $\mathbf{R}_{\beta\alpha} $ is either a $1 \times 1$ matrix or a $2 \times 2$ matrix. The $1 \times 1$ sub-matrix of $\mathbf{R}_{\beta\alpha} $ consisting of the row $ii$ and the column $ii$ is $\big[S_{\beta\alpha}(i)\big]$ and the $2 \times 2$ sub-matrix of $\mathbf{R}_{\beta\alpha} $ consisting of the rows $ij, ji$ and the columns  $ij,ji$ with $i<j$ is
     \begin{equation*}
     \left[
       \begin{array}{cc}
         S_{\beta\alpha}(i) & T_{\beta\alpha}(i) \\
         0 & -1 \\
       \end{array}
     \right].
     \end{equation*}
     Similarly, the $1 \times 1$ sub-matrix of $\mathbf{R}_{\alpha\beta}$ consisting of the row $ii$ and the column $ii$ is $\big[S_{\alpha\beta}(i)\big]$ and the $2 \times 2$ sub-matrix of $\mathbf{R}_{\alpha\beta}$ consisting of the rows $ij, ji$ and the columns  $ij,ji$ with $i<j$ is
     \begin{equation*}
     \left[
       \begin{array}{cc}
         S_{\alpha\beta}(i) & T_{\alpha\beta}(i) \\
         0 & -1 \\
       \end{array}
     \right].
     \end{equation*}
     It is trivial that $S_{\beta\alpha}(i)S_{\alpha\beta}(i) = 1$ and it can be verified
     \begin{equation*}
      \left[
       \begin{array}{cc}
         S_{\beta\alpha}(i) & T_{\beta\alpha}(i) \\
         0 & -1 \\
       \end{array}
     \right] \left[
       \begin{array}{cc}
         S_{\alpha\beta}(i) & T_{\alpha\beta}(i) \\
         0 & -1 \\
       \end{array}
     \right] = \left[
                 \begin{array}{cc}
                   1 & 0 \\
                   0 & 1 \\
                 \end{array}
               \right]
     \end{equation*}
       by the direct computation.\qed
\subsection{Initial condition}
The contour integral of (\ref{1017pm515}) multiplied by $\mathbf{J}(t)\prod_ib_{\nu(i)}^{-y_i}\prod_i\xi_i^{-y_i-1}$ from the left, that is, the right-hand side of (\ref{758pm512}) still satisfies (\ref{832pm515}) and (\ref{1019pm515}). (The contour is the one introduced in Theorem \ref{1249pm523}.) Hence, it remains to show that all transition probabilities $P_{(Y,\nu)}(X,\pi;t)$ satisfy the initial condition
\begin{eqnarray}
P_{(Y,\nu)}(X,\pi;0) &=~& \sum_{\sigma\in \mathcal{S}_N}\dashint_c\cdots \dashint_c {A}_{\sigma}^{\pi\nu}\prod_{i=1}^Nb_{\pi(i)}^{x_i}b_{\nu(i)}^{-y_i}\prod_{i=1}^N\Big(\xi_{\sigma(i)}^{x_i-y_{\sigma(i)}-1}\Big)~ d\xi_1\cdots d\xi_N \nonumber\\[5pt]
&=~&\begin{cases}
1 &~~\textrm{if $(Y,\nu) = (X,\pi)$;} \\
0 &~~\textrm{otherwise}
\end{cases}\label{159am527}
\end{eqnarray}
where ${A}_{\sigma}^{\pi\nu}$ is the $(\pi,\nu)$-element of $\mathbf{A}_{\sigma}$. We will show that the integral with the identity permutation in the sum satisfies (\ref{159am527}), and other integral terms with non-identity permutations are zero.
\begin{proof}
If $\sigma$ is the identity permutation, then  $\mathbf{A}_{\sigma}$ is the identity matrix. Hence, if $\pi \neq \nu$, then the integral is zero. It is easy to see that if $\pi = \nu$ and $x_i = y_i$ for all $i$, then the integral is  1, and if  $\pi = \nu$ and $x_i >  y_i$ for some $i$ (recall that our model is totally asymmetric), then the integral becomes zero when integrating with respect to $\xi_i$. Now, suppose that $\sigma$ is not the identity permutation. Note that the factors in ${A}_{\sigma}^{\pi\nu}$ are from (\ref{625pm724}), all poles arising from ${A}_{\sigma}^{\pi\nu}$, if any, are outside the contours. There exists an $i$ such that $x_i - y_{\sigma(i)}-1 \geq 0$ because each $x_i \geq y_i$ and $\sigma$ is not the identity permutation. Hence, integrating with respect to $i$, the integral is 0.
\end{proof}
 \section{Conclusion}
 In this paper, we have shown that the Bethe Ansatz method is still applicable to the multi-species TASEP with species-dependent rates. Theorem \ref{1249pm523} provides the transition probabilities for all possible compositions of species, which is expected to be used to study further objects such as the current distribution for some special initial configurations. The methods used in this paper have limitations to extend to the ASEP ($0<p<1$) for now, but it would be interesting to see if the methods can be used to study the species inhomogeneity of other multi-species models.
\\ \\
\textbf{Acknowledgement} This work was supported by the faculty development competitive research grant (090118FD5341) by Nazarbayev University.
\\ \\

\end{document}